\documentclass[12pt]{amsart}%
\usepackage{amsfonts}
\usepackage{amsmath}
\usepackage{amssymb}
\usepackage{graphicx}%
\newtheorem{theorem}{Theorem}
\theoremstyle{plain}

\newtheorem{corollary}{Corollary}

\newtheorem{definition}{Definition}

\newtheorem{lemma}{Lemma}

\newtheorem{proposition}{Proposition}
\newtheorem{remark}{Remark}

\numberwithin{equation}{section}
\begin{document}
\title[An integral identity and geometric applicaitons]{An integral formula in Kahler geometry with applications}
\author{Xiaodong Wang}
\address{Department of Mathematics, Michigan State University, East Lansing, MI 48824}
\email{xwang@math.msu.edu}

\begin{abstract}
We establish an integral formula on a smooth, precompact domain in a Kahler
manifold. We apply this formula to study holomorphic extension of CR
functions. Using this formula we prove an isoperimetric inequality in terms of
a positive lower bound for the Hermitian curvature of the boundary. Combining
with a Minkowski type formula on the complex hyperbolic space we prove that
any closed, embedded hypersurface of constant mean curvature must be a
geodesic sphere, provided the hypersurface is Hopf. A similar result is
established on the complex projective space.

\end{abstract}
\maketitle

\section{\bigskip Introduction}

Let $M$ be a compact Riemannian manifold with boundary. Given a real $u\in
C^{\infty}\left(  M\right)  $ the following formula was proved by Reilly
\cite{R}
\[
\begin{split}
 & \ \int_{M}\left[  \left(  \Delta u\right)  ^{2}-\left\vert D^{2}u\right\vert
^{2}-Ric\left(  \nabla u,\nabla u\right)  \right]  dv \\
=& \ \int_{\partial M}\left[
2\chi\Delta f+H\chi^{2}+\Pi\left(  \nabla f,\nabla f\right)  \right]  d\sigma,
\end{split}
\]
where $f=u|_{\partial M}$ and $\chi=\frac{\partial u}{\partial\nu}$. It was
used by Reilly \cite{R} to give a simple and elegant proof of the Alexandrov
theorem that a closed, embedded hypersurface of constant mean curvature in
$\mathbb{R}^{n}$ must be a round sphere. Since then Reilly's formula has
become a very useful tool in geometric analysis and is especially effective on
Riemannian manifolds with nonnegative Ricci curvature. For other applications,
see \cite{Ro} and the more recent \cite{MW}.

In this paper we derive a similar formula on a Kahler manifold. Let $M$ be a
Kahler manifold of complex dimension $m+1\geq2$. Suppose $\Omega\subset M$ is
a precompact domain with smooth boundary $\Sigma$. Then $\Sigma$ is a real
hypersurface in $M$. Let $\nu$ be the outer unit normal along $\Sigma$ and
$T=J\nu$. Let $Av=\nabla_{v}\nu$ be the shape operator, $\Pi\left(
v_{1},v_{2}\right)  =\left\langle Av_{1},v_{2}\right\rangle $ the second
fundamental form and $H$ the mean curvature (the trace of $\Pi$). The function
$H_{b}:=H-\Pi\left(  T,T\right)  $ will play a fundamental role in our
discussion. We call $H_{b}$ the Hermitian mean curvature of $\Sigma$. Set
$Z=\left(  \nu-\sqrt{-1}T\right)  /\sqrt{2}$, a $(1,0)$-vector field along
$\Sigma$. For a smooth function $F:\Omega\rightarrow\mathbb{C}$ we prove the
following integral formula%

\begin{equation*}
\begin{split}
& \ \sqrt{2}\int_{\Omega}\left\vert \square F\right\vert ^{2}-\left\vert
D^{1,1}F\right\vert ^{2} \\
= & \ \int_{\Sigma}\left[  Z\overline{F}\left(  \square_{b}f-\sqrt{-1}\Pi\left(
T,X_{\alpha}\right)  f_{\overline{\alpha}}\right)  +ZF\overline{\left(
\square_{b}f-\sqrt{-1}\Pi\left(  T,X_{\alpha}\right)  f_{\overline{\alpha}%
}\right)}  \right] \\
&+ \sqrt{2}\int_{\Sigma}\Pi\left(  \overline{X_{\alpha}},\overline{X_{\beta}%
}\right)  f_{\alpha}\overline{f}_{\beta}  +\frac{1}{\sqrt{2}}\int_{\Sigma}H_{b}\left\vert \overline{Z}F\right\vert
^{2},
\end{split}
\end{equation*}
where $f=F|_{\Sigma}$, $\square_{b}$ is the Kohn Laplacian and $\left\{
X_{\alpha}\right\}  $ a local unitary frame for $T^{1,0}\Sigma$.

A salient feature of the above identity is that the curvature of $M$ does not
appear in it. This makes it effective on a general Kahler manifold. Indeed, we
have found several applications for this formula.

The first application is about when one can extend a CR function on $\Sigma$
to a holomorphic function on $\Omega$. We are able to prove the following

\bigskip

\noindent
{\bf Theorem A.}
{\em Let $\Omega\subset M$ be a connected precompact domain with smooth boundary
$\Sigma$. Suppose $H_b>0$ on $\Sigma$. Then for any $f\in C^{\infty}\left(  \Sigma\right)  $ which is a CR
function, there exists $F\in C^{\infty}\left(  \overline{\Omega}\right)
\cap\mathcal{O}\left(  \Omega\right)  $ s.t. $F|_{\Sigma}=f$.}

Though this result is weaker that a theorem of Kohn-Rossi, our method has the
merit of being elementary.

Another application is the following geometric inequality.

\bigskip

\noindent
{\bf Theorem B.}
{\em Let $\Omega\subset M$ be a connected precompact domain with smooth boundary
$\Sigma$. If $H_{b}>0$ on $\Sigma$, then%
\[
\int_{\Sigma}\frac{1}{H_{b}}\geq\frac{m+1}{m}\left\vert \Omega\right\vert .
\] }

As a corollary we immediately obtain the following isoperimetric inequality in
terms of a positive lower bound for the Hermitian mean curvature.

\bigskip

\noindent
{\bf Theorem C.}
{\em Let $\Omega\subset M$ be a connected precompact domain with smooth boundary
$\Sigma$. Let $c=\inf_{\Sigma}H_{b}$. If $c>0$, then%
\begin{equation}
m\left\vert \Sigma\right\vert \geq c\left(  m+1\right)  \left\vert
\Omega\right\vert . \label{iso-int}%
\end{equation}
}

In the complex hyperbolic space $\mathbb{CH}^{m+1}$ or in the complex
projective space $\mathbb{CP}^{m+1}$ the above inequality is sharp as equality
holds if $\Omega$ is a geodesic ball. Conversely, we are able to prove

\bigskip

\noindent
{\bf Theorem D.}
{\em Let $\Omega\subset\mathbb{CH}^{m+1}$ be a connected precompact domain with
smooth boundary $\Sigma$. If equality holds in (\ref{iso-int}), then $\Omega$
is a geodesic ball.}

\bigskip The same result holds in $\mathbb{CP}^{m+1}$.

In view of the Alexandrov theorem mentioned before it is natural to study
closed hypersurfaces of constant mean curvature in other fundamental
Riemannian manifolds. Alexandrov proved the same uniqueness result for
$\mathbb{H}^{n}$ and $\mathbb{S}_{+}^{n}$ as well in his groundbreaking paper
\cite{A}. More recently, there have been some new spectacular developments.
Brendle \cite{Br1} proved a uniqueness result for hypersurfaces of constant
mean curvature in certain warped product manifolds of which the
deSitter-Schwarzschild spaces are important examples. In another paper
\cite{Br2} he proved that the Clifford torus is the only embedded minimal
torus in $\mathbb{S}^{3}$ up to congruence (the Lawson conjecture). Andrews
and Li \cite{AL} proved that all embedded tori of constant mean curvature in
$\mathbb{S}^{3}$ are rotationally symmetric and therefore completely classified.
Brendle \cite{Br3} further extended these results by proving that certain embedded Weingarten tori in $\mathbb{S}^{3}$,
which include constant mean curvature tori as a special case, must be rotationally symmetric. 

After $\mathbb{R}^{n},\mathbb{H}^{n}$ or $\mathbb{S}^{n}$, the most important
Riemannian manifolds are arguably the symmetric spaces. In a
symmetric space of rank $\geq2$, geodesic spheres do not have constant mean
curvature and it is not clear what to expect. In symmetric spaces of rank one
geodesic spheres do have constant mean curvature (and even constant principal
curvatures). It is natural to ask if they are the only closed, embedded hypersurfaces of
constant mean curvature. The above theorems yield some partial results on this open
problem in $\mathbb{CH}^{m+1}$ and $\mathbb{CP}^{m+1}$. Recall that a hypersurface $\Sigma$ in a Kahler manifold $M$ is
called Hopf if $T$ is an eigenvector of the shape operator $A$ at every point
of $\Sigma$, i.e. $AT=\alpha T$ with $\alpha=\Pi\left(  T,T\right)  $. It is a
well known fact that for a Hopf hypersurface in $\mathbb{CH}^{m+1}$ or in
$\mathbb{CP}^{m+1}$ the function $\alpha$ is constant. We can prove the
following results.

\bigskip

\noindent
{\bf Theorem E.}
{\em Let $\Sigma$ be a closed, embedded hypersurface in the complex hyperbolic
space $\mathbb{CH}^{m+1}$ with constant mean curvature. If $\Sigma$ is Hopf,
then it is a geodesic sphere.}

\bigskip

\noindent
{\bf Theorem F.}
{\em The same result holds for $\mathbb{CP}^{m+1}$ provided that $\Sigma$ is
disjoint from a hyperplane in $\mathbb{CP}^{m+1}$.}

\bigskip 
These results were established by Miquel \cite{M} under an extra condition.
More precisely, he needs to assume $\alpha\geq2\coth\left(  2\coth^{-1}\left(
H_{b}/2m\right)  \right)  $ in the case of $\mathbb{CH}^{m+1}$ and $\alpha
\geq2\cot\left(  2\cot^{-1}\left(  H_{b}/2m\right)  \right)  $ in the case of
$\mathbb{CP}^{m+1}$.

The paper is organized as follows. In Section 2 we prove the integral
identity. In Section 3 we discuss holomorphic extension of CR functions and
prove Theorem A. In Section 4 we present some geometric applications,
including Theorem B and Theorem C. In Section 5 we discuss real hypersurfaces
in $\mathbb{CH}^{m+1}$ and $\mathbb{CP}^{m+1}$ . Finally Theorems D, E and F
are proved in Section 6.

\textbf{Acknowledgement. }The author wishes to thanks Haizhong Li for fruitful
discussions and for his warm hospitality during his visit to Tsinghua
University in July, 2014. He also wants to thank Guofang Wei for her interest and comments.

\section{An integral formula on a domain in a Kahler manifold}

Let $M$ be a Kahler manifold of complex dimension $m+1\geq2$. \ We denote the
metric (extended as a complex bilinear form on the complexified tangent
bundle) by $\left\langle \cdot,\cdot\right\rangle $ and the Levi-Civita
connection by $\nabla$. Let $\Omega\subset M$ be a precompact domain with
smooth boundary. The boundary $\Sigma$ is endowed with the induced metric and
its outer unit normal is denoted by $\nu$. The Levi-Civita connection on
$\Sigma$ will be denoted by $\nabla^{\Sigma}$. The shape operator
$A:T\Sigma\rightarrow T\Sigma$ and the second fundamental form $\Pi$ of
$\Sigma$ are defined in the usual way: for $u,v\in T\Sigma$,
\begin{align*}
Au  &  =\nabla_{u}\nu,\\
\Pi\left(  u,v\right)   &  =\left\langle Au,v\right\rangle .
\end{align*}
Note that $T:=J\nu$ is a unit tangent vector field on $\Sigma$. Denote
$Z=\left(  \nu-\sqrt{-1}T\right)  /\sqrt{2}$.

There is a canonical CR structure on $\Sigma$: the distribution $\mathcal{H}%
=\left\{  u\in T\Sigma:\left\langle u,T\right\rangle =0\right\}  $ is
invariant under the complex structure $J$ and therefore $\mathcal{H}%
\otimes\mathbb{C=}T^{1,0}\Sigma\oplus T^{0,1}\Sigma$, where
\[
T^{1,0}\Sigma=\{u-\sqrt{-1}Ju:u\in\mathcal{H}\},T^{0,1}\Sigma=\overline
{T^{1,0}\Sigma}.
\]
We consider on $\Sigma$ the $1$-form $\theta=\left\langle T,\cdot\right\rangle
.$ Its Levi form $L$ is the Hermitian symmetric form on $T^{1,0}\Sigma$
defined by $L\left(  X,Y\right)  =-\sqrt{-1}d\theta\left(  X,\overline
{Y}\right)  $. A simple calculation yields%
\[
L\left(  X,Y\right)  =2\Pi\left(  X,\overline{Y}\right)  .
\]
Recall that $\Sigma$ is strictly pseudoconvex if the Levi form is positive
definite on $T^{1,0}\Sigma$. We denote by $H_{b}$ the trace of $L$ with
respect to the Hermitian metric on $T^{1,0}\Sigma$. A simple calculation shows
that
\[
H_{b}=H-\Pi\left(  T,T\right)  ,
\]
where $H$ is the mean curvature of $\Sigma$. In other words, $H_{b}$ is the
trace of $\Pi$ on the contact distribution $\mathcal{H}$. (The notation
$H_{b}$ is in analogy with the sub-Laplacian $\Delta_{b}$ in CR geometry,
which equals the trace of a Riemannian Hessian on the contact distribution
$\mathcal{H}$.)

For $F\in C^{\infty}\left(  \overline{\Omega}\right)  $ we denote by
$D^{1,1}F$ its complex Hessian and $\square F=-\overline{\partial}^{\ast
}\overline{\partial}F$ its complex Laplacian (which is half of the real
Laplacian). Let $\overline{\partial}_{b}:C^{\infty}\left(  \Sigma\right)
\rightarrow\mathcal{A}^{0,1}\left(  \Sigma\right)  $ be the tangential
Cauchy-Riemann operator, $\overline{\partial}_{b}^{\ast}:\mathcal{A}%
^{0,1}\left(  M\right)  \rightarrow C^{\infty}\left(  \Sigma\right)  $ its
dual and $\square_{b}=-\overline{\partial}_{b}^{\ast}\overline{\partial}_{b}$
the Kohn Laplacian. In doing computations we always work with a local unitary
frame $\left\{  X_{i}:0\leq i\leq m\right\}  $ for $T^{1,0}M$. We write%
\[
F_{i}=X_{i}F,F_{\overline{i}}=\overline{X_{i}}F,F_{i,\overline{j}}%
=X_{i}\overline{X_{j}}f-\nabla_{X_{i}}\overline{X_{j}}F
\]
etc. We will implicitly use the following rules%
\begin{align*}
F_{i,\overline{j}}  &  =F_{\overline{j},i},F_{i,\overline{j}\overline{k}%
}=F_{i,\overline{k}\overline{j}},\\
F_{i,\overline{j}k}  &  =F_{i,k\overline{j}}+R_{k\overline{j}i\overline{l}%
}F_{l}.
\end{align*}

Along $\Sigma$ we may and will assume that $X_{0}=Z$. Then $\left\{
X_{\alpha}:1\leq\alpha\leq m\right\}  $ is a local unitary frame for
$T^{1,0}\Sigma$. In the following Greek indices range from $1$ to $m$ while
Latin letters run form $0$ to $m$.

\begin{lemma}
\label{Kohn}For $f\in C^{\infty}\left(  \Sigma\right)  $ we have%
\[
\square_{b}f=X_{\alpha}\overline{X_{\alpha}}f-\left\langle \nabla_{X_{\alpha}%
}^{\Sigma}\overline{X_{\alpha}},X_{\beta}\right\rangle \overline{X_{\beta}%
}f+\sqrt{-1}\Pi\left(  T,X_{\alpha}\right)  \overline{X_{\alpha}}f.
\]

\end{lemma}

\begin{proof}
For $f,g\in C^{\infty}\left(  \Sigma\right)  $ we have%
\[
\left\langle \overline{\partial}_{b}f,\overline{\overline{\partial}_{b}%
g}\right\rangle =f_{\overline{\alpha}}\overline{g}_{\alpha}=X_{\alpha}\left(
f_{\overline{\alpha}}\overline{g}\right)  -\left(  X_{\alpha}f_{\overline
{\alpha}}\right)  \overline{g}.
\]
Consider the vector field $\Xi=\phi_{\overline{\alpha}}X_{\alpha}$ with
$\phi_{\overline{\alpha}}=f_{\overline{\alpha}}\overline{g}$. We compute%
\begin{align*}
\mathrm{div}\Xi &  =\left\langle \nabla_{X_{\beta}}^{\Sigma}\Xi,\overline
{X_{\beta}}\right\rangle +\left\langle \nabla_{T}^{\Sigma}\Xi,T\right\rangle
\\
&  =X_{\alpha}\phi_{\overline{\alpha}}-\left\langle \Xi,\nabla_{X_{\beta}%
}^{\Sigma}\overline{X_{\beta}}\right\rangle -\left\langle \Xi,\nabla
_{T}T\right\rangle \\
&  =X_{\alpha}\phi_{\overline{\alpha}}-\left\langle \nabla_{X_{\beta}%
}\overline{X_{\beta}},X_{\alpha}\right\rangle \phi_{\overline{\alpha}%
}-\left\langle J\nabla_{T}\nu,X_{\alpha}\right\rangle \phi_{\overline{\alpha}%
}\\
&  =X_{\alpha}\phi_{\overline{\alpha}}-\left\langle \nabla_{X_{\beta}%
}\overline{X_{\beta}},X_{\alpha}\right\rangle \phi_{\overline{\alpha}}%
+\sqrt{-1}\Pi\left(  T,X_{\alpha}\right)  \phi_{\overline{\alpha}}.
\end{align*}
Therefore%
\begin{align*}
\left\langle \overline{\partial}_{b}f,\overline{\overline{\partial}_{b}%
g}\right\rangle  &  =\mathrm{div}\Xi-\left[  \left(  X_{\alpha}f_{\overline
{\alpha}}\right)  \overline{g}-\left\langle \nabla_{X_{\beta}}\overline
{X_{\beta}},X_{\alpha}\right\rangle \phi_{\overline{\alpha}}+\sqrt{-1}%
\Pi\left(  T,X_{\alpha}\right)  \phi_{\overline{\alpha}}\right] \\
&  =\mathrm{div}\Xi-\left[  \left(  X_{\alpha}f_{\overline{\alpha}}\right)
-\left\langle \nabla_{X_{\beta}}\overline{X_{\beta}},X_{\alpha}\right\rangle
f_{\overline{\alpha}}+\sqrt{-1}\Pi\left(  T,X_{\alpha}\right)  f_{\overline
{\alpha}}\right]  \overline{g}%
\end{align*}
Integrating by parts yields%
\[
\int_{\Sigma}\left\langle \overline{\partial}_{b}f,\overline{\overline
{\partial}_{b}g}\right\rangle =-\int_{\Sigma}\left[  X_{\alpha}f_{\overline
{\alpha}}-\left\langle \nabla_{X_{\beta}}\overline{X_{\beta}},X_{\alpha
}\right\rangle f_{\overline{\alpha}}+\sqrt{-1}\Pi\left(  T,X_{\alpha}\right)
f_{\overline{\alpha}}\right]  \overline{g}.
\]

\end{proof}

\begin{lemma}
\bigskip We have $\mathrm{div}T=0$.
\end{lemma}

\begin{proof}
Since $T$ is of unit length, $\left\langle \nabla_{T}^{\Sigma}T,T\right\rangle
=0$. We compute%
\begin{align*}
\mathrm{div}T  &  =\left\langle \nabla_{X_{\alpha}}^{\Sigma}T,\overline
{X_{\alpha}}\right\rangle +\left\langle \nabla_{\overline{X_{\alpha}}}%
^{\Sigma}T,X_{\alpha}\right\rangle +\left\langle \nabla_{T}^{\Sigma
}T,T\right\rangle \\
&  =\left\langle \nabla_{X_{\alpha}}T,\overline{X_{\alpha}}\right\rangle
+\left\langle \nabla_{\overline{X_{\alpha}}}T,X_{\alpha}\right\rangle \\
&  =\left\langle J\nabla_{X_{\alpha}}\nu,\overline{X_{\alpha}}\right\rangle
+\left\langle J\nabla_{\overline{X_{\alpha}}}\nu,X_{\alpha}\right\rangle \\
&  =-\left\langle \nabla_{X_{\alpha}}\nu,J\overline{X_{\alpha}}\right\rangle
-\left\langle \nabla_{\overline{X_{\alpha}}}\nu,JX_{\alpha}\right\rangle \\
&  =\sqrt{-1}\left(  \left\langle \nabla_{X_{\alpha}}\nu,\overline{X_{\alpha}%
}\right\rangle -\left\langle \nabla_{\overline{X_{\alpha}}}\nu,X_{\alpha
}\right\rangle \right) \\
&  =\sqrt{-1}\left(  \Pi\left(  X_{\alpha},\overline{X_{\alpha}}\right)
-\Pi\left(  \overline{X_{\alpha}},X_{\alpha}\right)  \right) \\
&  =0.
\end{align*}

\end{proof}

\bigskip For later purposes we compare $\square_{b}f$ with $\Delta_{\Sigma}f$.

\begin{proposition}
For $f\in C^{\infty}\left(  \Sigma\right)  $ we have%
\begin{equation}
2\square_{b}f=\Delta_{\Sigma}f-D^{2}f\left(  T,T\right)  +\sqrt{-1}\left[
2\Pi\left(  T,X_{\alpha}\right)  \overline{X_{\alpha}}f-H_{b}Tf\right]  .
\label{compare}%
\end{equation}

\end{proposition}

\begin{proof}
From the previous lemma we have%
\begin{align*}
\square_{b}f  &  =X_{\alpha}\overline{X_{\alpha}}f-\nabla_{X_{\alpha}}%
^{\Sigma}\overline{X_{\alpha}}f+\left\langle \nabla_{X_{\alpha}}^{\Sigma
}\overline{X_{\alpha}},T\right\rangle Tf+\sqrt{-1}\Pi\left(  T,X_{\alpha
}\right)  \overline{X_{\alpha}}f\\
&  =D^{2}f\left(  X_{\alpha},\overline{X_{\alpha}}\right)  -\left\langle
\nabla_{X_{\alpha}}^{\Sigma}T,\overline{X_{\alpha}}\right\rangle Tf+\sqrt
{-1}\Pi\left(  T,X_{\alpha}\right)  \overline{X_{\alpha}}f\\
&  =D^{2}f\left(  X_{\alpha},\overline{X_{\alpha}}\right)  -\sqrt{-1}%
\Pi\left(  X_{\alpha},\overline{X_{\alpha}}\right)  Tf+\sqrt{-1}\Pi\left(
T,X_{\alpha}\right)  \overline{X_{\alpha}}f\\
&  =\frac{1}{2}\left(  \Delta_{\Sigma}f-D^{2}f\left(  T,T\right)  \right)
-\frac{\sqrt{-1}}{2}H_{b}Tf+\sqrt{-1}\Pi\left(  T,X_{\alpha}\right)
\overline{X_{\alpha}}f.
\end{align*}
This yields the desired identity.
\end{proof}

We now state again our integral formula in a Kahler manifold.

\begin{theorem}
For $F\in C^{\infty}\left(  \overline{\Omega}\right)  $ denote $f=F|_{\Sigma}%
$. Then%
\begin{equation}
\begin{split}
& \  \sqrt{2}\int_{\Omega}^{2}\left\vert \square F\right\vert ^{2}-\left\vert
D^{1,1}F\right\vert ^{2}\label{id}\\
= & \ \int_{\Sigma}\left[  Z\overline{F}\left(  \square_{b}f-\sqrt{-1}\Pi\left(
T,X_{\alpha}\right)  f_{\overline{\alpha}}\right)  +ZF\overline{\left(
\square_{b}f-\sqrt{-1}\Pi\left(  T,X_{\alpha}\right)  f_{\overline{\alpha}%
}\right)  } \right] \nonumber\\
& +\sqrt{2}\int_{\Sigma}\Pi\left(  \overline{X_{\alpha}},\overline{X_{\beta}%
}\right)  f_{\alpha}\overline{f}_{\beta} +\frac{1}{\sqrt{2}}\int_{\Sigma}H_{b}\left\vert \overline{Z}F\right\vert
^{2}.\nonumber
\end{split}
\end{equation}

\end{theorem}

\begin{proof}
\bigskip Working with a local unitary frame, we have%
\begin{align*}
\left\vert D^{1,1}F\right\vert ^{2}-\left\vert \square F\right\vert ^{2}  &
=F_{i\overline{j}}\overline{F}_{\overline{i}j}-\square F\overline
{F}_{j\overline{j}}\\
&  =\left(  F_{i\overline{j}}\overline{F}_{j}\right)  _{\overline{i}%
}-F_{i\overline{j},\overline{i}}\overline{F}_{j}-\square F\overline
{F}_{j\overline{j}}\\
&  =\left(  F_{i\overline{j}}\overline{F}_{j}\right)  _{\overline{i}}-\left(
\square F\right)  _{\overline{j}}\overline{F}_{j}-\square F\overline
{F}_{j\overline{j}}\\
&  =\left(  F_{i\overline{j}}\overline{F}_{j}\right)  _{\overline{i}}-\left(
\square F\overline{F}_{j}\right)  _{\overline{j}}.
\end{align*}
Integrating by parts we obtain%
\begin{align*}
\sqrt{2}\int_{\Omega}\left\vert D^{1,1}F\right\vert ^{2}-\left\vert \square
F\right\vert ^{2}  &  =\int_{\Sigma}D^{2}F\left(  Z,\overline{X_{j}}\right)
\overline{F}_{j}-\square F\left(  Z\overline{F}\right) \\
&  =\int_{\Sigma}D^{2}F\left(  Z,\overline{X_{\alpha}}\right)  \overline
{f}_{\alpha}-\left(  \square F-D^{2}F\left(  Z,\overline{Z}\right)  \right)
Z\overline{F}.
\end{align*}
We now analyze the boundary terms carefully. We compute on $\Sigma$ using
Lemma \ref{Kohn}%
\begin{align*}
\square F-D^{2}F\left(  Z,\overline{Z}\right)   &  =D^{2}F\left(  X_{\alpha
},\overline{X_{\alpha}}\right) \\
&  =X_{\alpha}\overline{X_{\alpha}}F-\nabla_{X_{\alpha}}\overline{X_{\alpha}%
}F\\
&  =X_{\alpha}\overline{X_{\alpha}}f-\left\langle \nabla_{X_{\alpha}}%
\overline{X_{\alpha}},X_{\beta}\right\rangle \overline{X_{\beta}%
}f-\left\langle \nabla_{X_{\alpha}}\overline{X_{\alpha}},Z\right\rangle
\overline{Z}F\\
&  =\square_{b}f-\sqrt{-1}\Pi\left(  T,X_{\alpha}\right)  \overline{X_{\alpha
}}f+\left\langle \nabla_{X_{\alpha}}Z,\overline{X_{\alpha}}\right\rangle
\overline{Z}F\\
&  =\square_{b}f-\sqrt{-1}\Pi\left(  T,X_{\alpha}\right)  \overline{X_{\alpha
}}f+\sqrt{2}\Pi\left(  X_{\alpha},\overline{X_{\alpha}}\right)  \overline
{Z}F\\
&  =\square_{b}f-\sqrt{-1}\Pi\left(  T,X_{\alpha}\right)  \overline{X_{\alpha
}}f+\frac{H_{b}}{\sqrt{2}}\overline{Z}F
\end{align*}
We compute on $\Sigma$
\begin{align*}
D^{2}F\left(  Z,\overline{X_{\alpha}}\right)   &  =\overline{X_{\alpha}%
}ZF-\nabla_{\overline{X_{\alpha}}}ZF\\
&  =\overline{X_{\alpha}}ZF-\left\langle \nabla_{\overline{X_{\alpha}}%
}Z,\overline{Z}\right\rangle ZF-\left\langle \nabla_{\overline{X_{\alpha}}%
}Z,\overline{X_{\beta}}\right\rangle X_{\beta}f\\
&  =\overline{X_{\alpha}}ZF-\sqrt{-1}\Pi\left(  T,\overline{X_{\alpha}%
}\right)  ZF-\sqrt{2}\Pi\left(  \overline{X_{\alpha}},\overline{X_{\beta}%
}\right)  X_{\beta}f.
\end{align*}
Therefore%
\begin{equation*}
\begin{split}
& \  \sqrt{2}\int_{\Omega}\left\vert D^{1,1}F\right\vert ^{2}-\left\vert \square
F\right\vert ^{2}\\
= &  \ \int_{\Sigma}\left(  \overline{X_{\alpha}}ZF-\sqrt{-1}\Pi\left(
T,\overline{X_{\alpha}}\right)  ZF-\sqrt{2}\Pi\left(  \overline{X_{\alpha}%
},\overline{X_{\beta}}\right)  f_{\beta}\right)  \overline{f}_{\alpha}\\
& \  +\int_{\Sigma}\left(  -\square_{b}f+\sqrt{-1}\Pi\left(  T,X_{\alpha
}\right)  \overline{X_{\alpha}}f-\frac{H_{b}}{\sqrt{2}}\right)  Z\overline
{F}\\
=&  \ \int_{\Sigma}\left(  -Z\overline{F}\square_{b}f-\overline{\square_{b}%
f}ZF+\sqrt{-1}\Pi\left(  T,X_{\alpha}\right)  f_{\overline{\alpha}}%
Z\overline{F}-\sqrt{-1}\Pi\left(  T,\overline{X_{\alpha}}\right)  \overline
{f}_{\alpha}ZF\right) \\
& -\sqrt{2}\int_{\Sigma}\Pi\left(  \overline{X_{\alpha}},\overline{X_{\beta}%
}\right)  \overline{f}_{\alpha}f_{\beta} -\frac{1}{\sqrt{2}}\int_{\Sigma}H_{b}\left\vert \overline{Z}F\right\vert
^{2},
\end{split}
\end{equation*}
where in the process we did integration by part on $\Sigma$. Reorganizing the
terms yields (\ref{id}).
\end{proof}

\section{\bigskip Holomorphic extension of CR functions}

Let $M$ be a complex manifold of complex dimension $m+1\geq2$ and
$\Sigma\subset M$ a real hypersurface. A function $f$ on $\Sigma$ is called CR
if it satisfies the tangential Cauchy-Riemann equations, i.e. $\overline
{X}f=0$ for all $X\in T^{1,0}\Sigma$. Obviously, a holomorphic function on a
neighborhood of $\Sigma$ restricts to a CR function on $\Sigma$. Conversely,
it is an interesting question if all CR functions arise this way. If $\Sigma$
encloses a domain $\Omega$, one can also ask the global question: does a CR
function on $\Sigma$ extends to a holomorphic function on $\Omega$?

On this problem, we have the following classic result (Theorem 2.3.2' in
Hormander \cite{H})

\begin{theorem}
Let $\Omega$ be a smooth, bounded open set in
$\mathbb{C}^{m+1},m\geq1$, s.t. $\mathbb{C}^{m+1}\backslash\overline{\Omega}%
$ is connected. If $u\in C^{\infty}\left(  \partial\Omega\right)
$ is a CR function, one can find a holomorphic function $U\in
C^{\infty}\left(  \overline{\Omega}\right)  $ s.t. $U=u\,$%
 on $\partial\Omega$\textit{.}
\end{theorem}

In a general complex manifold, Kohn and Rossi \cite{KR} proved the following

\begin{theorem}
Let $\Omega$\textit{ be a precompact domain with smooth boundary in a complex
manifold }$M^{m+1}$. Suppose the boundary is connected and the Levi form on
the boundary has one positive eigenvalue everywhere, then every $CR$ function
on $\partial\Omega$ has a holomorphic extension to $\Omega$.
\end{theorem}

Their approach is via the solution of the $\overline{\partial}$-Neumann
problem, and particularly the regularity of solutions at the boundary.

Using the formula (\ref{id}), we present an elementary new approach to this
problem. The basic idea is very simple: given a CR function $f$ on $\Sigma$
one should try to prove its harmonic extension on $\Omega$ is holomorphic. Our
approach works in any Kahler manifold with a mild pointwise condition on
$\Sigma$. The result we can prove is weaker than the Kohn-Rossi theorem, but
the method has the merit of being elementary. In particular we avoid the
analytically sophisticated $\overline{\partial}$-Neumann problem.

Let $M$ be a Kahler manifold of complex dimension $m+1\geq2$. Let
$\Omega\subset M$ be a (connected) precompact domain with smooth boundary
$\Sigma$. For simplicity we assume everything is smooth. But the optimal
regularity required for the method to work should be obvious.

\begin{theorem}\label{ext}
Suppose that $\Sigma$ satisfies the following positivity condition%
\begin{equation}
H_{b}>0. \label{pos}%
\end{equation}
The for any $f\in C^{\infty}\left(  \Sigma\right)  $ which is a CR function,
there exists $F\in C^{\infty}\left(  \overline{\Omega}\right)  \cap
\mathcal{O}\left(  \Omega\right)  $ s.t. $F|_{\Sigma}=f$.
\end{theorem}

\begin{remark}
Condition (\ref{pos}) is much weaker than strict pseudoconvexity.
\end{remark}

Let $F\in C^{\infty}\left(  \overline{\Omega}\right)  $ be the harmonic
extension of $f$. By the integral identity%
\[
\int_{\Omega}\left\vert D^{1,1}F\right\vert ^{2}=-\frac{1}{4}\int_{\Sigma
}H_{b}\left\vert \overline{Z}F\right\vert ^{2}.
\]
Under the boundary condition we must have $D^{1,1}F=0$ on $\Omega$ and
$\overline{Z}F=0$ on $\Sigma$. Integrating by parts we have%
\begin{align*}
\int_{\Omega}\left\vert \overline{\partial}F\right\vert ^{2}  &  =\int
_{\Omega}F_{\overline{j}}\overline{F}_{j}\\
&  =\int_{\Omega}\left(  F_{\overline{j}}\overline{F}\right)  _{j}\\
&  =\int_{\Sigma}\left(  \overline{Z}F\right)  \overline{f}\\
&  =0.
\end{align*}
Therefore $F$ is holomorphic.

\begin{corollary}
\bigskip Under the same assumption, any holomorphic function on $M\backslash
\overline{\Omega}$ extends to a holomorphic function on $M$.
\end{corollary}

If we only assume that $H_{b}\geq0$, then the argument above shows that the
harmonic extension $F$ is pluriharmoinc: $F_{i\overline{j}}=0$. Is it possible
to prove that $F$ is holomorphic? Or equivalently, is Theorem \ref{ext} valid under the
condition $H_{b}\geq0$?

\section{Geometric inequalities}

As a first application of the integral formula (\ref{id}), we prove the following

\begin{proposition}
Let $\Omega$ be a connected precompact domain with smooth boundary $\Sigma$ in
a Kahler manifold. If $\Sigma$ satisfies $H_{b}>0$, then $\Sigma$ is connected.
\end{proposition}

\begin{proof}
Suppose $\Sigma$ is not connected and let $\Sigma_{1}$ be a connected
component and $\Sigma_{2}=\Sigma\backslash\Sigma_{1}$. Let $f\in C^{\infty
}\left(  \Sigma\right)  $ be the function that is $1$ on $\Sigma_{1}$ and $0$
on $\Sigma_{2}$. Let $u\in C^{\infty}\left(  \overline{\Omega}\right)  $ be
the harmonic extension of $f$ and $\chi=\frac{\partial u}{\partial\nu}$. Note
that $u$ is real. By the maximum principle and Hopf Lemma,%
\[
0<u<1\text{ on }\Omega\text{; }\chi>0\text{ on }\Sigma_{1}\text{; }%
\chi<0\text{ on }\Sigma_{2}.
\]
By (\ref{id}) we have%
\[
\int_{\Omega}\left\vert D^{1,1}u\right\vert ^{2}=-\frac{1}{4}\int_{\Sigma
}H_{b}\chi^{2}<0,
\]
a contradiction.
\end{proof}

\begin{remark}
\bigskip It is interesting to compare our Proposition with the following
classic fact in Riemannian geometry: a compact connected Riemannian manifold
with mean convex boundary and nonnegative Ricci curvature has at most two
boundary components; moreover if \ $\partial M$ has two components, then $M$
is isometric to a cylinder $N\times\left[  0,a\right]  $ over some connected
closed Riemannian manifold $N$ with nonnegative Ricci curvature (cf.
\cite{I}). Note that in the Kahler case we do not impose any curvature
assumption on $\Omega$. The proof here is similar to the proof in \cite{HW} of
the aforementioned fact in Riemannian geometry using Reilly's formula.
\end{remark}

\bigskip We now consider $F\in C^{\infty}\left(  \overline{\Omega}\right)  $
which is the solution of the following boundary value problem%
\begin{equation}
\left\{
\begin{array}
[c]{ccc}%
\square F=\left(  m+1\right)  & \text{on} & \overline{\Omega},\\
F=0 & \text{on} & \Sigma.
\end{array}
\right.  \label{bv}%
\end{equation}

\bigskip Note that $F$ is real. Denote $\chi=\frac{\partial F}{\partial\nu}$.
By the strong maximum principle and the Hopf Lemma%
\[
F<0\text{ on }\Omega\text{; }\chi>0\text{ on }\Sigma\text{.}%
\]

\begin{theorem}
\label{invHb}Let $\Omega\subset M$ be a connected precompact domain with
smooth boundary $\Sigma$. If $H_{b}>0$ on $\Sigma$, then%
\begin{equation}
\int_{\Sigma}\frac{1}{H_{b}}\geq\frac{m+1}{m}\left\vert \Omega\right\vert .
\label{invHb-F}%
\end{equation}

\end{theorem}

\begin{proof}
Integrating the equation (\ref{bv}) yields%
\begin{align*}
\left(  m+1\right)  \left\vert \Omega\right\vert  &  =\frac{1}{2}\int_{\Sigma
}\chi\\
&  \leq\frac{1}{2}\left(  \int_{\Sigma}H_{b}\chi^{2}\right)  ^{1/2}\left(
\int_{\Sigma}\frac{1}{H_{b}}\right)  ^{1/2}.
\end{align*}
Thus
\begin{equation}
\frac{1}{4}\int_{\Sigma}H_{b}\chi^{2}\geq\left(  m+1\right)  ^{2}\left\vert
\Omega\right\vert ^{2}/\int_{\Sigma}\frac{1}{H_{b}}. \label{intc}%
\end{equation}
By the integral identity (\ref{id}) applied to $F$ (noting $\overline
{Z}F=\frac{1}{\sqrt{2}}H_{b}$) we obtain%
\begin{align*}
\frac{1}{4}\int_{\Sigma}H_{b}\chi^{2}  &  =\int_{\Omega}\left\vert \square
F\right\vert ^{2}-\left\vert D^{1,1}F\right\vert ^{2}\\
&  \leq\left(  1-\frac{1}{m+1}\right)  \int_{\Omega}\left\vert \square
F\right\vert ^{2}\\
&  =\left(  m+1\right)  m\left\vert \Omega\right\vert .
\end{align*}
Combined with (\ref{intc}) this implies
\[
\int_{\Sigma}\frac{1}{H_{b}}\geq\frac{m+1}{m}\left\vert \Omega\right\vert .
\]

\end{proof}

\begin{remark}
\bigskip The theorem and its proof are similar to the following result of Ros
\cite{R}:

Let $(M^{n},g)$ be a compact Riemannian manifold with boundary. If $Ric\geq0$
and the mean curvature $H$ of $\partial M$ is positive, then%
\[
\int_{\partial M}\frac{1}{H}d\sigma\geq\frac{n}{n-1}V.
\]
The equality holds iff $M$ is isometric to an Euclidean ball.

But Theorem \ref{invHb} is valid in any Khaler manifold: there is no curvature
assumption on $\Omega$.
\end{remark}

\begin{remark}
It is clear from the proof that if equality holds in (\ref{invHb-F}) we must
have $D^{1,1}F=I$ and $H_{b}\chi=a$ (a constant) on $\Sigma$. In fact the
first identity implies the second (see below). Therefore on any Kahler
manifold if $F$ is a (local) Kahler potential and $c$ a regular value for $F$
s.t. $\left\{  F\leq c\right\}  $ is compact, then for $\Omega=\left\{
F<c\right\}  $ we have equality in (\ref{invHb-F}).
\end{remark}

As an immediate corollary, we obtain the following isoperimetric inequality in
terms of a positive lower bound for $H_{b}$.

\begin{theorem}
\label{iso}\bigskip Let $\Omega\subset M$ be a connected precompact domain
with smooth boundary $\Sigma$. Let $c=\inf_{\Sigma}H_{b}$. If $c>0$, then%
\[
m\left\vert \Sigma\right\vert \geq c\left(  m+1\right)  \left\vert
\Omega\right\vert .
\]

\end{theorem}

\bigskip

We now analyze the equality case in Theorem \ref{iso}. From the proof we must have

\begin{enumerate}
\item $D^{1,1}F=I$, i.e. for $\left(  1,0\right)  $-vectors $X\,,Y$%
\begin{equation}
D^{2}F\left(  X,\overline{Y}\right)  =\left\langle X,\overline{Y}\right\rangle
. \label{HF}%
\end{equation}

\item $H_{b}\equiv c$ and $\chi$ is a positive constant.
\end{enumerate}

\bigskip

\begin{lemma}
On $\Sigma$ we have $\chi c=2m$ and for $X\,,Y\in T^{1,0}\Sigma$
\begin{align*}
\Pi\left(  X,\overline{Y}\right)   &  =\frac{c}{2m}\left\langle X,\overline
{Y}\right\rangle ,\\
\Pi\left(  T,X\right)   &  =0.
\end{align*}

\end{lemma}

\begin{proof}
Working with a unitary frame along $\Sigma$ we have by (\ref{HF})%
\begin{align*}
\delta_{\alpha\beta}  &  =F_{\alpha\overline{\beta}}\\
&  =X_{\alpha}\overline{X_{\beta}}F-\left(  \nabla_{X_{\alpha}}\overline
{X_{\beta}}\right)  F\\
&  =\left\langle \nabla_{X_{\alpha}}\nu,\overline{X_{\beta}}\right\rangle
\chi\\
&  =\Pi\left(  X_{\alpha},\overline{X_{\beta}}\right)  \chi,
\end{align*}
as $F|_{\Sigma}=0$. Therefore $\Pi\left(  X_{\alpha},\overline{X_{\beta}%
}\right)  =\frac{1}{\chi}\delta_{\alpha\beta}$. Taking trace yields $\chi
c=2m$. Similarly, as $\chi$ is constant%
\begin{align*}
0  &  =F_{\alpha\overline{0}}\\
&  =X_{\alpha}\overline{Z}F-\left(  \nabla_{X_{\alpha}}\overline{Z}\right)
F\\
&  =X_{\alpha}\chi-\left\langle \nabla_{X_{\alpha}}\overline{Z},\nu
\right\rangle \chi\\
&  =\left\langle \nabla_{X_{\alpha}}\nu,\overline{Z}\right\rangle \chi\\
&  =\frac{\sqrt{-1}}{\sqrt{2}}\Pi\left(  X_{\alpha},T\right)  \chi.
\end{align*}
This implies $\Pi\left(  X_{\alpha},T\right)  =0$.
\end{proof}

\begin{remark}
\bigskip Since $c>0$ the 1st identity implies that $\Sigma$ is strictly pseudoconvex.
\end{remark}

Let $A$ be the shape operator on $\Sigma$, i.e. $A:T\Sigma\rightarrow T\Sigma$
is the symmetric endomorphism defined by $Av=\nabla_{v}\nu$. We have
$\Pi\left(  u,v\right)  =\left\langle Au,v\right\rangle $. In the first
identity if we take $X=u-\sqrt{-1}Ju,Y=v-\sqrt{-1}Jv$, we obtain for any
$u,v\in\mathcal{H}$%
\[
\left\langle Au,v\right\rangle +\left\langle AJu,Jv\right\rangle =\frac{c}%
{m}\left\langle u,v\right\rangle .
\]
In other words, restricted on $\mathcal{H}$ we have%
\begin{equation}
A-JAJ=\frac{c}{m}I. \label{jaj}%
\end{equation}
The second identity in the above lemma implies $AT=\alpha T$, where
$\alpha=\Pi\left(  T,T\right)  $. $\ $

\begin{definition}
A hypersurface $\Sigma$ in a complex manifold is called\ is called a Hopf
hypersurface if $T$ is an eigenvector of the shape operator at every point of
$\Sigma$.
\end{definition}

Such hypersurfaces have been studied intensively in $\mathbb{CP}^{m+1}$ and
$\mathbb{CH}^{m+1}$, cf. Niebergall-Ryan \cite{NR} for a detailed survey of
the subject. We will have further discussion on Hopf hypersurfaces in the next section.

\section{\bigskip Hopf hypersurfaces in $\mathbb{CH}^{m+1}$ and $\mathbb{CP}%
^{m+1}$}

In this Section we discuss hypersurfaces in $\mathbb{CP}^{m+1}$ and
$\mathbb{CH}^{m+1}$. For basic facts on $\mathbb{CP}^{m+1}$ and $\mathbb{CH}%
^{m+1}$ we refer to \cite{KN}. We will take an intrinsic approach. The
starting point is that they are the unique simply connected, Kahler manifold
of constant holomorphic sectional curvature $4\kappa$, with $\kappa=-1$ for
$\mathbb{CH}^{m+1}$ and $\kappa=1$ for $\mathbb{CP}^{m+1}$. It follows that
the sectional curvature of a $J$-invariant $2$-plane is $4\kappa$ while the
sectional curvature of a totally real $2$-plane is $\kappa$. The curvature
tensor is explicitly given by%
\begin{align}
R\left(  v_{1},v_{2},v_{3},v_{4}\right)   &  =\kappa\left[  \left\langle
v_{1},v_{3}\right\rangle \left\langle v_{2},v_{4}\right\rangle -\left\langle
v_{1},v_{4}\right\rangle \left\langle v_{2},v_{3}\right\rangle \right.
\label{curv}\\
&  \left.  +\left\langle v_{1},Jv_{3}\right\rangle \left\langle v_{2}%
,Jv_{4}\right\rangle -\left\langle v_{1},Jv_{4}\right\rangle \left\langle
v_{2},Jv_{3}\right\rangle +2\left\langle v_{1},Jv_{2}\right\rangle
\left\langle v_{3},Jv_{4}\right\rangle \right]  .\nonumber
\end{align}
Equivalently, for $\left(  1,0\right)  $-vectors $X,Y,Z,W$%
\[
R\left(  X,\overline{Y},Z,\overline{W}\right)  =-2\kappa\left[  \left\langle
X,\overline{Y}\right\rangle \left\langle Z,\overline{W}\right\rangle
-\left\langle X,\overline{W}\right\rangle \left\langle Z,\overline
{Y}\right\rangle \right]  .
\]

We now focus on $\mathbb{CH}^{m+1}$. Fix a point $o\in\mathbb{CH}^{m+1}$ and
let $r$ be the distance function to $o$. The Hessian of $r$ is given by%
\begin{align}
\nabla_{\nabla r}\nabla r  &  =0,\nabla_{J\nabla r}\nabla r=2\frac{\cosh
2r}{\sinh2r}J\nabla r,\label{Hess}\\
\nabla_{v}\nabla r  &  =\frac{\cosh r}{\sinh r}v,\text{ if }\left\langle
v,\nabla r\right\rangle =\left\langle v,J\nabla r\right\rangle =0.\nonumber
\end{align}

From (\ref{Hess}) it is also clear that on a geodesic sphere of radius $a>0$
the shape operator $A$ has the following form%
\[
AT=2\frac{\cosh2a}{\sinh2a}T,A|_{\mathcal{H}}=\frac{\cosh a}{\sinh a}I.
\]

Let $\Phi=\log\cosh r$. It is a smooth function on $\mathbb{CH}^{m+1}$. A
straightforward calculation shows that $D^{\left(  1,1\right)  }\Phi=I$, i.e.
for any $\left(  1,0\right)  $-vectors $X,Y$%
\[
D^{2}\Phi\left(  X,\overline{Y}\right)  =\left\langle X,\overline
{Y}\right\rangle .
\]
In particular $\square\Phi=m+1$.

\begin{remark}
The existence of $F$ is also clear from the ball model of $\mathbb{CH}^{m+1}$:
the unit ball in $\mathbb{C}^{m+1}$ with the Kahler form%
\[
\omega=-\frac{\sqrt{-1}}{2}\partial\overline{\partial}\log\left(  1-\left\vert
z\right\vert ^{2}\right)  .
\]

\end{remark}

As a simple consequence of the existence of $F$, we prove the following
Minkowski type formula (cf. \cite{M}, where it is derived in a different way).

\begin{proposition}
\label{MFP}Suppose $\Sigma$ is a closed Hopf hypersurface in $\mathbb{CH}%
^{m+1}$. Then%
\begin{equation}
2m\left\vert \Sigma\right\vert =\int_{\Sigma}H_{b}\left\langle \nabla\Phi
,\nu\right\rangle . \label{MF}%
\end{equation}

\end{proposition}

\begin{proof}
Let $\Phi=\Phi|_{\Sigma},\chi=\frac{\partial\Phi}{\partial\nu}=\left\langle
\nabla\Phi,\nu\right\rangle $. Then%
\[
2\left(  m+1\right)  =\Delta\Phi|_{\Sigma}=\Delta_{\Sigma}\Phi+H\chi+D^{2}%
\Phi\left(  \nu,\nu\right)  .
\]
For $Z=\left(  \nu-\sqrt{-1}T\right)  /\sqrt{2}$ we have%
\[
1=D^{2}\Phi\left(  Z,\overline{Z}\right)  =\frac{1}{2}\left[  D^{2}\Phi\left(
\nu,\nu\right)  +D^{2}\Phi\left(  T,T\right)  \right]  .
\]
On the other hand%
\begin{align*}
D^{2}\Phi\left(  T,T\right)   &  =TT\Phi-\left(  \nabla_{T}T\right)  \Phi\\
&  =TT\Phi-\left(  \nabla_{T}^{\Sigma}T\right)  \Phi+\Pi\left(  T,T\right)
\chi\\
&  =D^{2}\Phi\left(  T,T\right)  +\Pi\left(  T,T\right)  \chi.
\end{align*}
Combining these identities yields
\[
2m=\Delta_{\Sigma}\Phi-D^{2}\Phi\left(  T,T\right)  +H_{b}\chi.
\]
Integrating over $\Sigma$ yields
\[
2m\left\vert \Sigma\right\vert =\int_{\Sigma}H_{b}\left\langle \nabla\Phi
,\nu\right\rangle -\int_{\Sigma}D^{2}\Phi\left(  T,T\right)  .
\]
Integrating the identity (\ref{compare}) yields%
\[
\int_{\Sigma}D^{2}\Phi\left(  T,T\right)  =2\operatorname{Re}\int_{\Sigma
}\sqrt{-1}\Pi\left(  T,X_{\alpha}\right)  \overline{X_{\alpha}}\Phi.
\]
Since $\Sigma$ is Hopf, the right hand side is clearly zero. Thus we obtain
(\ref{MF}).
\end{proof}

\begin{remark}
On $\mathbb{C}^{m+1}$ consider $\Phi=\left\vert z\right\vert ^{2}/2$. First
recall the classic Minkowski formula: If $\Sigma\subset\mathbb{C}^{m+1}$ is a
closed hypersurface then%
\[
\left(  2m+1\right)  \left\vert \Sigma\right\vert =\int_{\Sigma}H\left\langle
\nabla\Phi,\nu\right\rangle
\]
By the same argument used to prove Proposition (\ref{MFP}) we have the
following formula%
\[
2m\left\vert \Sigma\right\vert =\int_{\Sigma}H_{b}\left\langle \nabla\Phi
,\nu\right\rangle ,
\]
provided that $\Sigma$ is Hopf.
\end{remark}

\bigskip

\bigskip We now discuss the case of $\mathbb{CP}^{m+1}$ endowed with the
Fubini-Study metric normalized to have holomorphic sectional curvature $4$.
The diameter is $\pi/2$. For any $o=\left[  \xi\right]  \in\mathbb{CP}^{m+1}$
($\xi\in\mathbb{C}^{m+2}\backslash\left\{  0\right\}  $) the set of points
whose distance to $o$ equals $\pi/2$ is also its cut-locus
\[
C\left(  o\right)  :=\left\{  \left[  \zeta\right]  \in\mathbb{CP}%
^{m+1}:\left\langle \zeta,\overline{\xi}\right\rangle =0\right\}  ,
\]
which is a totally geodesic $\mathbb{CP}^{m}$ , called a hyperplane. On the
domain
\[
\mathbb{CP}^{m+1}\backslash C\left(  o\right)  =\left\{  p\in\mathbb{CP}%
^{m+1}:d\left(  o,p\right)  <\pi/2\right\}
\]
the function $\Phi=\log\cos r$ is smooth and satisfies
\[
D^{2}\Phi\left(  X,\overline{Y}\right)  =\left\langle X,\overline
{Y}\right\rangle .
\]
for any $\left(  1,0\right)  $-vectors $X,Y$. Therefore $\Delta\Phi=2\square
F=2\left(  m+1\right)  $.

Therefore as in the complex hyperbolic case we have

\begin{proposition}
\label{MFP1}Suppose $\Sigma$ is a closed Hopf hypersurface in $\mathbb{CH}%
^{m}$. If there exists $o\in\mathbb{CP}^{m+1}$ s.t. $\Sigma\cap C\left(
o\right)  =\varnothing$, then%
\[
2m\left\vert \Sigma\right\vert =\int_{\Sigma}H_{b}\left\langle \nabla\Phi
,\nu\right\rangle .
\]

\end{proposition}

Again, this formula was first proved by Miquel \cite{M}.

Similar to the complex hyperbolic case, on a geodesic sphere of radius
$a\in\left(  0,\pi/2\right)  $ in $\mathbb{CP}^{m+1}$ the shape operator $A$
has the following form%
\[
AT=2\frac{\cos2a}{\sin2a}T,A|_{\mathcal{H}}=\frac{\cos a}{\sin a}I.
\]
More generally let $\Sigma$ be a tube over a totally geodesic $\mathbb{CP}%
^{k}$, i.e.
\[
\Sigma=\left\{  p\in\mathbb{CP}^{m+1}:d\left(  p,\mathbb{CP}^{k}\right)
=a\right\}
\]
for some $a\in\left(  0,\pi/2\right)  $. Then the shape operator $A$ has three eigenvalues:

\begin{itemize}
\item $\lambda_{1}=2\cos2a/\sin2a$ of multiplicity $1$,

\item $\lambda_{2}=\cos a/\sin a$ of multiplicity $2\left(  m-k\right)  $,

\item $\lambda_{3}=-\sin a/\cos a$ of multiplicity $2k$.
\end{itemize}

\bigskip

\section{\bigskip Uniqueness results in $\mathbb{CH}^{m+1}$ and $\mathbb{CP}%
^{m+1}$}

We now prove the following rigidity result.

\begin{theorem}
\label{rig}Let $M$ be a simply connected Kahler manifold of constant
holomorphic sectional curvature $4\kappa$ with $\dim_{\mathbb{C}}M=m+1\geq2$.
In other words $M=\mathbb{CP}^{m+1}$ if $\kappa=1$; $M=\mathbb{CH}^{m+1}$ if
$\kappa=-1$. Let $\Omega\subset M$ be a connected precompact domain with
smooth boundary $\Sigma$. Let $c=\inf_{\Sigma}H_{b}$. If \ equality holds in
Theorem \ref{iso}, i.e. $m\left\vert \Sigma\right\vert =c\left(  m+1\right)
\left\vert \Omega\right\vert $, then $\Omega$ is a geodesic ball.
\end{theorem}

Suppose $\Sigma$ has constant Hermitian mean curvature $c$. Let $\Omega$ be
the domain enclosed by $\Sigma$ (in the case of $\mathbb{CP}^{m+1}$ we choose
$\Omega$ to be the one disjoint from the hyperplane in the assumption). By
Proposition \ref{MFP} or Proposition \ref{MFP1}, Since $m\left\vert
\Sigma\right\vert =c\left(  m+1\right)  \left\vert \Omega\right\vert $, we
have $c>0$. By the discussion in Section 4, we know that there exists $F\in
C^{\infty}\left(  \overline{\Omega}\right)  $ s.t. $F=0$ on $\Sigma$ and
$D^{1,1}F=I$. Moreover $\Sigma$ is a Hopf hypersurface,
\begin{equation}
AT=\alpha T \label{Hopf}%
\end{equation}
and its shape operator satisfies on $\mathcal{H}$
\begin{equation}
A-JAJ=\frac{c}{m}I. \label{ajaj}%
\end{equation}

\begin{lemma}
The function $\alpha$ is constant.
\end{lemma}

\begin{proof}
This is a well know fact on Hopf hypersurfaces in complex spaces forms, cf.
\cite{NR}. We provide a direct proof. Let $X$ be a tangential vector field on
$\Sigma$ s.t. $\left\langle X,T\right\rangle =0$. Differentiating (\ref{Hopf})
yields%
\begin{align*}
X\alpha &  =\left\langle \nabla_{X}\nabla_{T}\nu,T\right\rangle \\
&  =\left\langle \nabla_{T}\nabla_{X}\nu,T\right\rangle -\left\langle
\nabla_{\left[  T,X\right]  }\nu,T\right\rangle +R\left(  T,X,\nu,T\right)  .
\end{align*}
From (\ref{Hopf}) we obtain
\begin{align*}
\nabla_{T}T  &  =J\nabla_{T}v\\
&  =JAT\\
&  =-\alpha\nu.
\end{align*}
Thus%
\begin{align*}
\left\langle \nabla_{T}\nabla_{X}\nu,T\right\rangle  &  =T\left\langle
\nabla_{X}\nu,T\right\rangle -\left\langle \nabla_{X}\nu,\nabla_{T}%
T\right\rangle \\
&  =T\Pi\left(  T,X\right)  -\alpha\left\langle \nabla_{X}\nu,\nu\right\rangle
\\
&  =0.
\end{align*}
\ On the other hand%
\begin{align*}
\left\langle \nabla_{\left[  T,X\right]  }\nu,T\right\rangle  &  =\left\langle
\nabla_{T}\nu,\left[  T,X\right]  \right\rangle \\
&  =\alpha\left\langle T,\left[  T,X\right]  \right\rangle \\
&  =\alpha\left\langle T,\nabla_{T}X-\nabla_{X}T\right\rangle \\
&  =\alpha\left\langle T,\nabla_{T}X\right\rangle \\
&  =\alpha\left(  T\left\langle X,T\right\rangle -\left\langle \nabla
_{T}T,X\right\rangle \right) \\
&  =0.
\end{align*}
By the formula for curvature, $R\left(  T,X,\nu,T\right)  =0$. Therefore
$X\alpha=0$.

Since $\Sigma$ is strictly pseudoconvex, $T\alpha=0$ as well. Therefore
$\alpha$ is constant.
\end{proof}

For further discussion, we need the following fact on Hopf hypersurfaces in
complex space forms (Lemma 2.2. in Niebergall and Ryan \cite{NR}). Let
$\varphi:T\Sigma\rightarrow T\Sigma$ be the endomorphism s.t. $\varphi T=0$
and $\varphi|_{\mathcal{H}}=J$.

\begin{lemma}
Let $\Sigma$ be a Hopf hypersurface in a Kahler manifold $M$ of constant
holomorphic sectional curvature $\kappa$. Then
\begin{equation}
A\varphi A-\frac{a}{2}\left(  A\varphi+\varphi A\right)  =\kappa\varphi.
\label{Cod}%
\end{equation}

\end{lemma}

This is in fact a simple consequence of the Codazzi equation. We sketch the
proof here. By direct calculation, for any tangent vector $\left(  \nabla
_{X}^{\Sigma}A\right)  T=\left(  aI-A\right)  \varphi AX$. Thus%
\[
\left\langle \left(  \nabla_{X}^{\Sigma}A\right)  Y,T\right\rangle
=\left\langle \left(  \nabla_{X}^{\Sigma}A\right)  T,Y\right\rangle
=\left\langle \left(  aI-A\right)  \varphi AX,Y\right\rangle .
\]
By the Codazzi equation%
\begin{align*}
\left\langle \left(  aI-A\right)  \varphi AX,Y\right\rangle -\left\langle
\left(  aI-A\right)  \varphi AY,X\right\rangle  &  =\left\langle \left(
\nabla_{X}^{\Sigma}A\right)  Y,T\right\rangle -\left\langle \left(  \nabla
_{Y}^{\Sigma}A\right)  X,T\right\rangle \\
&  =R\left(  X,Y,T,\nu\right) \\
&  =2\kappa\left\langle X,\varphi Y\right\rangle ,
\end{align*}
where in the last step we used the curvature formula (\ref{curv}). The
identity (\ref{Cod}) follows easily.

Suppose $u\in\mathcal{H}$ is an eigenvector of $A$, $Au=\lambda u$. From
(\ref{ajaj}) we easily obtain $AJu=\left(  \frac{c}{m}-\lambda\right)  Ju$,
i.e. $Ju$ is also an eigenvector with eigenvalue $\frac{c}{m}-\lambda$.
Applying (\ref{Cod}) yields%
\begin{equation}
\lambda\left(  \frac{c}{m}-\lambda\right)  =\frac{\alpha c}{2m}+\kappa
.\label{qr}%
\end{equation}
This means that besides $\alpha$ the principal curvatures of $\Sigma$ can only
take at most two values, the two roots $\lambda$ and $\lambda^{\ast}=\frac
{c}{m}-\lambda$ of the quadratic equation
\[
x\left(  \frac{c}{m}-x\right)  =\frac{\alpha c}{2m}+\kappa.
\]
Therefore $\Sigma$ is a Hopf hypersurface with constant principal
curvatures.\ Such hypersurfaces in $\mathbb{CP}^{m+1}$ and $\mathbb{CH}^{m+1}$
are completely classified (even locally) by Kimura \cite{K} and Berndt
\cite{B}. \ We could finish the proof of Theorem \ref{rig} by doing a case by
case analysis of the classification list. But there is a more direct approach
which avoids using the classification. All we need is a fundamental formula
from \cite{B}. (Except this formula our proof is self-contained.)

There are two possibilities:

\begin{enumerate}
\item \bigskip The two roots coincide $\lambda=\lambda^{\ast}$.

In this case we have $\lambda=c/2m$ and $A=\lambda I$ on $\mathcal{H}$.

\item The two roots are different $\lambda\neq\lambda^{\ast}$.

In this case we have an orthogonal decomposition $\mathcal{H}=E\oplus JE$,
where $E$ is a real subspace of dimension $m$ and with respect to this
decomposition $A$ is given by the matrix%
\[
\left[
\begin{array}
[c]{cc}%
\lambda I & 0\\
0 & \lambda^{\ast}I
\end{array}
\right]  .
\]

By the fundamental formula in Berndt \cite[Theorem 2]{B} we must have
\begin{equation}
\lambda\lambda^{\ast}+\kappa=0. \label{fundB}%
\end{equation}

\end{enumerate}

For further discussion we discuss the two cases $\mathbb{CH}^{m+1}$ and
$\mathbb{CP}^{m+1}$ separately.

\begin{itemize}
\item $M=\mathbb{CH}^{m+1}$
\end{itemize}

\begin{lemma}
\label{comp}$\alpha>2$ and $\lambda,\lambda^{\ast}>1$.
\end{lemma}

\begin{proof}
This is a simple comparison. Let $o$ be a point enclosed by $\Sigma$ and
$\rho$ the distance function to $o$ on $\mathbb{CH}^{m+1}$. Let $p\in\Sigma$
be a farthest point on $\Sigma$ to $o$ and $a=d\left(  o,p\right)  $. Then at
$p$ we have $\nu=\nabla\rho$ and for any $X\in T_{p}\Sigma$%
\[
\Pi\left(  X,X\right)  \geq D^{2}\rho\left(  X,X\right)  .
\]
By (\ref{Hess}) we have by taking either $X=T$ or $u\in\mathcal{H}$ in the
above inequality
\begin{align*}
\alpha &  \geq2\frac{\cosh2a}{\sinh2a},\\
A  &  \geq\frac{\cosh a}{\sinh a}I\text{ on }\mathcal{H}.
\end{align*}
The second inequality implies that $\lambda$ and $\lambda^{\ast}$ are at least
$\frac{\cosh a}{\sinh a}>1$.
\end{proof}

\begin{lemma}
Let $r>0$ be the number s.t. $\alpha=2\frac{\cosh2r}{\sinh2r}$. Then
$A=\frac{\cosh r}{\sinh r}I$ on $\mathcal{H}$.
\end{lemma}

\begin{proof}
If $\lambda\neq\lambda^{\ast}$ we would have by (\ref{fundB}) $\lambda
\lambda^{\ast}=1$. But by the previous Lemma all eigenvalues are greater than
$1$. Therefore $\lambda=\lambda^{\ast}=\frac{c}{2m}$ and $A=\frac{c}{2m}I$ on
$\mathcal{H}$. Then (\ref{qr}) becomes%
\[
\left(  \frac{c}{2m}\right)  ^{2}=\frac{\alpha c}{2m}-1.
\]
If $\alpha=2\frac{\cosh2r}{\sinh2r}$, then we can easily obtain from the above
equation
\[
\frac{c}{2m}=\frac{\cosh r}{\sinh r}%
\]
as the other root $\sinh r/\cosh r<1$ must be discarded.
\end{proof}

We can now finish the proof. For $a>0$ consider the map $\Phi_{a}%
:\Sigma\rightarrow\mathbb{CH}^{m+1}$ defined by%
\[
\Phi_{a}\left(  p\right)  =\exp_{p}\left(  -a\nu\left(  p\right)  \right)  .
\]
By solving the Jacobi equation along the geodesic $\gamma_{p}\left(  t\right)
=\exp_{p}\left(  -a\nu\left(  p\right)  \right)  $ we have for $u\in
T_{p}\Sigma$%
\[
\left(  \Phi_{a}\right)  _{\ast}u=\left\{
\begin{array}
[c]{cc}%
\left(  \cosh2a-\frac{\cosh2r}{\sinh2r}\sinh2a\right)  U\left(  a\right)  , &
\text{if }u=T\text{,}\\
\left(  \cosh2a-\frac{\cosh r}{\sinh r}\sinh2a\right)  U\left(  a\right)  &
\text{if }u\in\mathcal{H}\text{,}%
\end{array}
\right.
\]
where $U\left(  t\right)  $ denotes the parallel vector field along
$\gamma_{p}$ with $U\left(  0\right)  =u$.

This shows that $\Phi_{r}$ is fully degenerate and hence maps $\Sigma$ to a
point $o$. Therefore $\Sigma$ is the geodesic sphere with center $o$ and
radius $r$.

\begin{itemize}
\item $M=\mathbb{CP}^{m+1}$
\end{itemize}

The discussion is parallel. Let $o\in\mathbb{CP}^{m+1}$ s.t. $\Sigma
\subset\mathbb{CP}^{m+1}\backslash C\left(  o\right)  $. Let $p\in\Sigma$ be a
point on $\Sigma$ s.t. $a:=d\left(  o,p\right)  =\max_{x\in\Sigma}d\left(
o,x\right)  \in\left(  0,\pi/2\right)  $. By a comparison argument similar to
the proof of Lemma \ref{comp} we see that $A\geq\frac{\cos a}{\sin a}I$ on
$\mathcal{H}$ at $p$. This implies that both $\lambda$ and $\lambda^{\ast}$
are positive. As $\kappa=1$ the identity (\ref{fundB}) is impossible.
Therefore $\lambda=\lambda^{\ast}=\frac{c}{2m}$ and $A=\frac{c}{2m}I$ on
$\mathcal{H}$.

Let $r\in\left(  0,\pi/2\right)  $ be the number s.t. $\lambda=\cos r/\sin r$.
From (\ref{qr}) we easily obtain $\alpha=2\cos2r/\sin2r$. By a similar
argument as in the hyperbolic case we conclude that $\Sigma$ is a geodesic
sphere of radius $r$.

As a corollary, we obtain the following uniqueness theorem for Hopf
hypersurfaces of constant mean curvature.

\begin{theorem}
Let $M$ be a simply connected Kahler manifold of constant holomorphic
sectional curvature $4\kappa$ with $\dim_{\mathbb{C}}M=m+1\geq2$. In other
words $M=\mathbb{CP}^{m+1}$ if $\kappa=1$; $M=\mathbb{CH}^{m+1}$ if
$\kappa=-1$. Let $\Sigma$ be a closed, embedded hypersurface in $M$. When
$M=\mathbb{CP}^{m+1}$ we further assume that $\Sigma$ is disjoint from a
hyperplane. If $\Sigma$ has constant mean curvature and is Hopf, then it is a
geodesic sphere.
\end{theorem}

\begin{remark}
\bigskip A tube over a totally geodesic $\mathbb{CP}^{k}$ ($0<k<m$) in
$\mathbb{CP}^{m+1}$ discussed above shows that the extra condition that
$\Sigma$ is disjoint from a hyperplane is necessary.
\end{remark}

\begin{proof}
Since $\Sigma$ is Hopf, $\alpha=\Pi\left(  T,T\right)  $ is constant. As
$H_{b}=H-\alpha\,$\ and $H$ is constant, we see that $H_{b}$ is constant.
Suppose $H_{b}=c\,$. Let $\Omega$ be the domain enclosed by $\Sigma$ (in the
case of $\mathbb{CP}^{m+1}$ we take $\Omega$ to be the one disjoint from the
same hyperplane). Since $\Sigma$ is Hopf we have by Proposition \ref{MFP} and
Proposition \ref{MFP1} (using the same notation there)%
\begin{align*}
2m\left\vert \Sigma\right\vert  &  =\int_{\Sigma}H_{b}\left\langle \nabla
\Phi,\nu\right\rangle \\
&  =c\int_{\Sigma}\left\langle \nabla\Phi,\nu\right\rangle \\
&  =c\int_{\Omega}\Delta\Phi\\
&  =2\left(  m+1\right)  c\left\vert \Omega\right\vert ,
\end{align*}
i.e. $m\left\vert \Sigma\right\vert =\left(  m+1\right)  c\left\vert
\Omega\right\vert $. Therefore $\Sigma$ is a geodesic sphere by Theorem
\ref{rig}.
\end{proof}


\begin{thebibliography}{999}                                                                                              %


\bibitem[A]{A}A. D. Alexandrov, Uniqueness theorems for surfaces in the large.
V. (Russian. English summary) Vestnik Leningrad. Univ. 13 1958 no. 19, 5--8.

\bibitem[AL]{AL}B. Andrews and H. Li, Embedded constant mean curvature tori in
the three-sphere. arXiv: 1204.5007

\bibitem[B]{B}J. Berndt, Real hypersurfaces with constant principal curvatures
in complex hyperbolic space. J. Reine Angew. Math. 395 (1989), 132--141.

\bibitem[Br1]{Br1}S. Brendle, Constant mean curvature surfaces in warped
product manifolds. Publ. Math. Inst. Hautes \'{E}tudes Sci. 117 (2013), 247--269.

\bibitem[Br2]{Br2}S. Brendle, Embedded minimal tori in S3 and the Lawson
conjecture. Acta Math. 211 (2013), no. 2, 177--190.

\bibitem[Br3]{Br3}S. Brendle, Embedded Weingarten tori in $S^3$.
Adv. Math. 257 (2014), 462--475. 

\bibitem[CR]{CR}T. Cecil and P. Ryan, Focal sets and real hypersurfaces in
complex projective space. Trans. Amer. Math. Soc. 269 (1982), 481--498.

\bibitem[HW]{HW}F. Hang and X. Wang, Vanishing sectional curvature on the
boundary and a conjecture of Schroeder and Strake. Pacific J. Math. 232
(2007), no. 2, 283--287.

\bibitem[H]{H}L. H\"{o}rmander, An introduction to complex analysis in several
variables. Third Edition (Revised). North-Holland Mathematical Library, Vol.
7. North-Holland Publishing Co.

\bibitem[I]{I}R. Ichida, Riemannian manifolds with compact boundary. Yokohama
Math. J. 29 (1981), no. 2, 169--177.

\bibitem[K]{K}M. Kimura, Real hypersurfaces and complex submanifolds in
complex projective space. Trans. Amer. Math. Soc. 296 (1986), no. 1, 137--149.

\bibitem[KN]{KN}S. Kobayashi; K. Nomizu, Foundations of differential geometry.
Vol. II. Reprint of the 1969 original. Wiley Classics Library. A
Wiley-Interscience Publication. John Wiley \& Sons, Inc., New York, 1996.

\bibitem[KR]{KR}K. J. Kohn, J. J.; H. Rossi, On the extension of holomorphic
functions from the boundary of a complex manifold. Ann. of Math. (2) 81 (1965), 451--472.

\bibitem[MW]{MW}P. Miao and X. Wang, Boundary effect of Ricci curvature. Preprint.

\bibitem[M]{M}V. Miquel, Compact Hopf hypersurfaces of constant mean curvature
in complex space forms. Ann. Global Anal. Geom. 12 (1994), no. 3, 211--218.

\bibitem[MR]{MR}S. Montiel and A. Ros, Compact hypersurfaces: the Alexandrov
theorem for higher order mean curvatures, in H. Blaine Lawson Jr., K.
Tenenblat (eds.), Differential Geometry, Pitman Monographs and Surveys in Pure
and Applied Mathematics, vol. 52, pp. 279--296, Longman, Harlow, 1991

\bibitem[NR]{NR}R. Niebergall; P. Ryan, Real hypersurfaces in complex space
forms. Tight and taut submanifolds (Berkeley, CA, 1994), 233--305, Math. Sci.
Res. Inst. Publ., 32, Cambridge Univ. Press, Cambridge, 1997.

\bibitem[R]{R}R. Reilly, Applications of the Hessian operator in a Riemannian
manifold, Indiana Univ. Math. J. 26 (1977), 459--472.

\bibitem[Ro]{Ro}A. Ros, Compact hypersurfaces with constant higher order mean
curvatures. Rev. Mat. Iberoamericana 3 (1987), no. 3-4, 447--453.
\end{thebibliography}
\end{document}